\documentclass{amsart}

\usepackage{hyperref}

\usepackage{cleveref}

\Crefname{equation}{Eq.}{Eqs.}

\usepackage{url}

\usepackage{tikz}
\usepackage{amsfonts}

\allowdisplaybreaks

\newcommand{\F}{\mathcal{F}}
\newcommand{\Z}{\mathbb{Z}}

\newcommand{\abs}{\mathrel{\unlhd}}
\newcommand{\Ring}{E[[x]]}

\newcommand{\R}{\mathbb{R}}
\newcommand{\N}{\mathbb{N}}
\newcommand{\DiZ}{E[[x]]}

\DeclareMathOperator{\od}{od}
\DeclareMathOperator{\ev}{ev}

\newtheorem{definition}{Definition}
\newtheorem{proposition}{Proposition}
\newtheorem{lemma}{Lemma}
\newtheorem{theorem}{Theorem}
\newtheorem{corollary}{Corollary}

\title[Breaks in log-concavity of spherically symmetric trees]{Unbounded log-concavity breaks in independence polynomials of spherically symmetric trees}

\author{C\'esar Bautista-Ramos}
\address{Fac. Ciencias de la Computaci\'on, Benem\'erita Universidad Aut\'onoma de Puebla, 14 sur y Av. San Claudio, CCO3-304, Ciudad Universitaria, Puebla Pue. 72570, Mexico 
}

\email{cesar.bautista@correo.buap.mx}

\begin{document}

\begin{abstract}
  Using a dioid algebraic structure, we show that there exist spherically symmetric trees $T(2^m 1^n)$ whose independence polynomials exhibit multiple breaks in log-concavity, a result established by estimating the asymptotic growth of the coefficients of these polynomials. Provided the parameter $n$ is a sufficiently large odd integer, the number of breaks is bounded below by the Jacobsthal numbers. This result affirmatively answers a question raised by F. Bencs [{D.~Galvin}, \emph {Trees with non log-concave independent set sequences}, arXiv 2502.10654.v2].

\end{abstract}

%\subjclass[2020]{05C31, 05A20, 16Y60, 05C05, 05C69}

\maketitle

\section{Introduction}
In 1987, P. Erd\H{o}s and his coauthors posed the question of whether the independent set sequence of any tree is unimodal \cite{Alavi}. Despite numerous efforts, this question remains unanswered to this day. Let us recall that the independent set sequence of a graph is formed by the coefficients of its independence polynomial, which can be defined inductively as follows \cite{Arocha, Gutman}: given a simple graph $G$ and a vertex $v$ of $G$, the \emph{independence polynomial of $G$}, denoted by $I(G;x)$, satisfies the recurrence relation
\begin{equation}\label{eq:Aro1}
	I(G;x)=I(G\backslash v;x)+xI(G\backslash N[v];x),
\end{equation}
where $N[v]$ denotes the closed neighborhood of $v$ and the backslash denotes vertex deletion. Furthermore, if $G$ is the disjoint union of two graphs $G_1$ and $G_2$, then
\begin{equation}\label{eq:Aro2}
	I(G;x)=I(G_1;x)I(G_2;x),
\end{equation}
with the initial condition $I(K_0;x)=1$, where $K_0$ is the graph with no vertices.

Although it is customary to define an independent set at this stage, we omit this definition to underscore our focus on the purely algebraic aspects of the independence polynomial, which follow directly from \Cref{eq:Aro1,eq:Aro2}. We refer the interested reader to \cite{Galvin, Hopkins, Kadrawi2, Levit} and the references therein. Recall that a finite sequence of real numbers $(a_i)_{i=0}^d$ is \emph{unimodal} if there exists an index $k$ such that $a_0 \le a_1 \le \dots \le a_k \ge a_{k+1} \ge \dots \ge a_d$, and it is \emph{log-concave} if $a_i^2 \ge a_{i-1}a_{i+1}$ for all $1 \le i \le d-1$.

Since log-concavity implies unimodality for nonnegative sequences without internal zeros and is often easier to prove, many authors establish unimodality via log-concavity. In fact, V.~E. Levit and E. Mandrescu conjectured that the independence sequence of any tree is log-concave \cite{Levit}. This conjecture was later disproved by O. Kadrawi et al.~\cite{Kadrawi}, with further counterexamples provided in \cite{Galvin,Kadrawi2}. A common feature of these early counterexamples is that log-concavity seems to fail at exactly one index. A natural question is whether there exist trees whose independence sequences exhibit multiple breaks in log-concavity. Inspired by the tree constructions in \cite{Galvin}, the present author proved that an arbitrary number of breaks is possible \cite{yo}. Furthermore, F.~Bencs noted that another type of tree, the \emph{spherically symmetric tree} denoted by $T(2^m1^n)$, exhibits this behavior for relatively small values of $m$ and $n$ \cite{Galvin}. It is worth mentioning that spherically symmetric trees arise in probability theory and statistics, particularly in the study of random walks \cite{Chiclana,Lyons}.

Building on these observations, F. Bencs asked whether there exist trees $T(2^m1^n)$ with an arbitrary number of breaks in their independence sequences \cite[Question 3.2]{Galvin}. This paper answers this question in the affirmative. More precisely, it is shown that for any $m \geq 3$ and any sufficiently large odd integer $n$, the independence sequence of $T(2^m1^n)$ has at least $J_{m-3}$ breaks in log-concavity, where $J_m$ denotes the $m$-th Jacobsthal number \cite{Horadam} (see \Cref{c:final}). Establishing this result requires approximately solving the non-linear recurrence that arises from~\Cref{eq:Aro1,eq:Aro2} for this specific case. 

Our analysis focuses on the algebraic structure of the exact asymptotic order of growth, as $n\to\infty$, of the coefficients of the independence polynomials $I(T(2^m1^n);x)$. This algebraic structure is known as a \emph{dioid} \cite{Gondran}. The use of dioids for studying diverse phenomena ranges from nonlinear partial differential equations in physics \cite{Gondran}, through the tropicalization of algebraic geometry \cite{Itenberg}, to modeling manufacturing processes \cite{Brunsch}.  In this paper, we use a dioid based on the exact asymptotic order $\Theta$ as an algebraic framework for solving the nonlinear recurrence satisfied by the independence polynomials of $T(2^m 1^n)$ in terms of novel \emph{deformed Jacobsthal polynomials}, which generalize the aforementioned Jacobsthal numbers. This approach proves particularly useful because these deformed polynomials readily exhibit multiple breaks in log-concavity (see \Cref{l:breaks}).

The remainder of this paper proceeds as follows. \Cref{s:defs} establishes the necessary definitions, notation, and algebraic framework. \Cref{s:Jacobsthal} introduces the deformed Jacobsthal polynomials and examines their absorption algebra, laying the groundwork to establish an idempotent property for these polynomials in \Cref{s:geo} (\Cref{thm:idem}). Finally, \Cref{s:JacSst} applies these algebraic results to the independence polynomials of spherically symmetric trees $T(2^m 1^n)$ (\Cref{thm:main}). This final analysis demonstrates that these polynomials can exhibit an arbitrary number of breaks in log-concavity (see \Cref{c:final}).

\section{Definitions, notation, and algebraic framework}\label{s:defs} 
As usual, $\N$ and $\Z$ denote the sets of nonnegative integers and integers, respectively. The symbol $\R_{\geq 0}$ denotes the set of nonnegative real numbers, and $P_n$ represents the path graph with $n$ vertices.
Recall that a rooted graph $G$ with root $r$ is \emph{spherically symmetric} if the following condition holds: all vertices at the same distance from the root $r$ have the same degree. The trees defined below provide examples of spherically symmetric graphs.

\begin{definition}
Let $m$ and $n$ be nonnegative integers. The tree $T(2^m1^n)$ is defined as the $m$-level perfect binary tree with a copy of the path $P_n$ dangling from each of its leaves (see \Cref{fig:T}).
\end{definition}

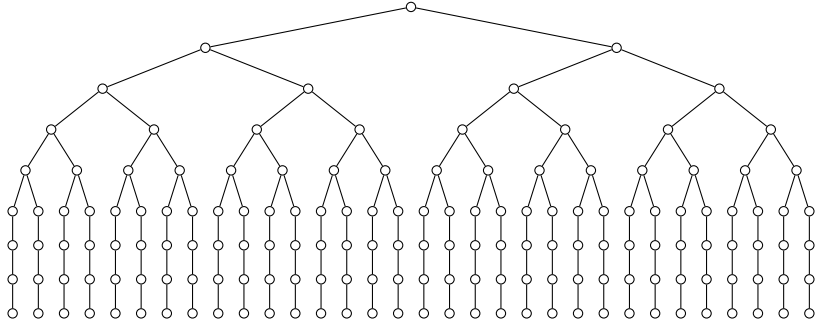
\begin{figure}
\centering

\begin{tikzpicture}[xscale=0.85, yscale=0.6,
    vertex/.style={circle, draw, fill=white, minimum size=3.5pt, inner sep=0pt},
    edge/.style={draw, thin}
]

  % --- Level 0: Root ---
  \node[vertex] (N-0-0) at (0, 0) {};

  % --- Level 1 (2 nodes) ---
  \foreach \k in {0,1} {
    \pgfmathsetmacro{\x}{(-1 + 2*\k) * 3.2}
    \node[vertex] (N-1-\k) at (\x, -0.9) {};
    \draw[edge] (N-0-0) -- (N-1-\k);
  }

  % --- Level 2 (4 nodes) ---
  \foreach \k in {0,...,3} {
    \pgfmathsetmacro{\x}{(-3 + 2*\k) * 1.6}
    \pgfmathtruncatemacro{\p}{floor(\k/2)}
    \node[vertex] (N-2-\k) at (\x, -1.8) {};
    \draw[edge] (N-1-\p) -- (N-2-\k);
  }

  % --- Level 3 (8 nodes) ---
  \foreach \k in {0,...,7} {
    \pgfmathsetmacro{\x}{(-7 + 2*\k) * 0.8}
    \pgfmathtruncatemacro{\p}{floor(\k/2)}
    \node[vertex] (N-3-\k) at (\x, -2.7) {};
    \draw[edge] (N-2-\p) -- (N-3-\k);
  }

  % --- Level 4 (16 nodes) ---
  \foreach \k in {0,...,15} {
    \pgfmathsetmacro{\x}{(-15 + 2*\k) * 0.4}
    \pgfmathtruncatemacro{\p}{floor(\k/2)}
    \node[vertex] (N-4-\k) at (\x, -3.6) {};
    \draw[edge] (N-3-\p) -- (N-4-\k);
  }

  % --- Level 5: Binary Core Leaves (32 nodes) ---
  \foreach \k in {0,...,31} {
    \pgfmathsetmacro{\x}{(-31 + 2*\k) * 0.2}
    \pgfmathtruncatemacro{\p}{floor(\k/2)}
    \node[vertex] (N-5-\k) at (\x, -4.5) {};
    \draw[edge] (N-4-\p) -- (N-5-\k);
  }

  % --- Pending Paths P_4 (3 additional nodes extended below each leaf) ---
  \foreach \k in {0,...,31} {
    \pgfmathsetmacro{\x}{(-31 + 2*\k) * 0.2}
    \node[vertex] (P1-\k) at (\x, -5.25) {};
    \node[vertex] (P2-\k) at (\x, -6.00) {};
    \node[vertex] (P3-\k) at (\x, -6.75) {};

    \draw[edge] (N-5-\k) -- (P1-\k);
    \draw[edge] (P1-\k) -- (P2-\k);
    \draw[edge] (P2-\k) -- (P3-\k);
  }

\end{tikzpicture}
\caption{The spherically symmetric tree $T(2^5 1^3)$. A copy of the path $P_3$ dangles from each of the $2^5$ leaves of the $5$-level perfect binary tree.}
  \label{fig:T}
\end{figure}

Let $t_{m,n}$ denote the independence polynomial of the spherically symmetric tree $T(2^m1^n)$. Then, the recursive formulas in \Cref{eq:Aro1,eq:Aro2} yield the following non-linear recurrence:
\begin{equation}
\label{eq:Aro}
t_{m,n}=t_{m-1,n}^2+x\,t_{m-2,n}^4, \quad m\geq 2,
\end{equation}
with initial conditions
\[
t_{0,n}=I(P_{n+1};x),\quad t_{1,n}=I(P_{2n+3};x).
\]

Assume that $n$ is odd and write $n=2k+1$. The primary strategy involves solving the recurrence \eqref{eq:Aro} approximately by considering only the exact asymptotic order of the coefficients of $t_{m,2k+1}$ as $k\to\infty$. The following discussion establishes the appropriate algebraic framework to develop this approach.

Let $\F$ be the semiring of functions $f(k) \colon \N \to \R_{\geq 0}$ under pointwise addition and multiplication. The exact asymptotic order $\Theta$ is a \emph{congruence relation} on $\F$, inducing a semiring structure on the quotient $E = \F/\Theta$ \cite[pp.~95--96]{Golan}. Furthermore, since $f(k) + f(k) \in \Theta(f(k))$ for all $f \in \F$, $E$ is a \emph{dioid} \cite{Gondran}. This naturally extends to a dioid structure on the formal power series $E[[x]]$ \cite[Sec.~2.2, p.~54]{Gondran}, which,  like any dioid, carries a \emph{canonical} partial order defined by the following absorption relation.

\begin{definition}
  Given $p, q \in \Ring$, the series $p$ \emph{is absorbed by} $q$ (or $q$ \emph{absorbs} $p$) if $p+q =q$. In such a case, we write $p \abs q$.
\end{definition}

The following properties follow directly from the definition.
\begin{proposition}\label{p:Aabs} Let $p, q, r, s \in \Ring$. Then:
  \begin{enumerate}
  \item\label{p:Aabs1} $p\abs q$ implies $p+r\abs q+r$ and $pr\abs qr$;
  \item\label{p:Aabs1a} $p\abs q$ and $r\abs s$ imply $p+r\abs q+s$ and $pr\abs qs$;
  \item\label{p:Aabs2} if $q=p+r$ then $p\abs q$;
  \item\label{p:Aabs3} if $p\abs q$ and $r\abs q$ then $p+r\abs q$;
    \item\label{p:Aabs4} if $p\abs q$ then $p\abs r+q$.
  \end{enumerate}
\end{proposition}

In the subsequent analysis, we make extensive use of \Cref{p:Aabs} without further warning. 

\section{Deformed Jacobsthal polynomials and their absorption algebra}\label{s:Jacobsthal}
This section introduces the deformed Jacobsthal polynomials, which form the central focus of this analysis. These polynomials are defined because they capture the asymptotic growth of the coefficients of the independence polynomials of the trees $T(2^m1^n)$ as $n\to\infty$, as shown in \Cref{thm:main}. In fact, this result follows from an idempotent property of the Jacobsthal polynomials within the dioid $E[[x]]$. This idempotent property constitutes the main result of this section (\Cref{thm:idem}). The discussion begins by recalling the Jacobsthal numbers \cite{Horadam}.

\begin{definition}\label{def:2}
  \begin{enumerate}
  \item The \emph{Jacobsthal sequence} $(J_m)_{m \geq 0}$ is defined by the recurrence
    \[
      J_m = J_{m-1} + 2J_{m-2}, \quad m \geq 2,
    \]
    with the initial conditions $J_0 = 0$ and $J_1 = 1$.
  \item For any integer $m \geq 0$, let $K_m = k^{2^m} x^{J_m} \in \Ring$.
  \end{enumerate}
\end{definition}

The Binet-type formula for the Jacobsthal numbers is given by
\begin{equation}\label{eq:cldJ}
	J_m=\frac{2^m-(-1)^m}{3},
\end{equation}
which readily establishes the following frequently used identities for the monomials $K_m$:
\begin{equation}\label{eq:Ks}
K_{m}=K_{m-1}K_{m-2}^2 \quad \text{and} \quad x^{\ev(m)}K_m=x^{\od(m)}K_{m-1}^2,
\end{equation}
where the parity functions $\od$ and $\ev$ are defined as follows.

\begin{definition}
The parity functions $\od \colon \N \to \{0,1\}$ and $\ev \colon \N \to \{0,1\}$ are defined by
\[
\od(m)=
\begin{cases}
  1, &\text{if $m$ is odd;}\\
  0, &\text{if $m$ is even,}
\end{cases} \quad \text{and} \quad \ev(m)=1-\od(m).
\]
\end{definition}

The preceding definitions serve as auxiliary elements to define the deformed Jacobsthal polynomials.

\begin{definition}[Deformed Jacobsthal polynomials]
  Let $(p_m)_{m\geq 0}$ be a sequence of polynomials in $\Ring$, defined recursively by
  \[
    p_m = K_{m-2}^{2\od(m)}p_{m-1} + K_{m-1}^{\ev(m)}(1+K_{m-2})p_{m-2}, \quad m \geq 2,
  \]
  with initial conditions $p_0=0$ and $p_1=1$.
\end{definition}

Note that evaluating the polynomials $p_m$ at $x=k=1$ recovers the Jacobsthal numbers. Setting $x=0$ and applying induction on $m$ yields $p_m=1$ for all $m\geq 1$; that is, for $m\geq 1$, the polynomial $p_m$ has a constant term of $1$. Furthermore, \Cref{p:Aabs} \Cref{p:Aabs2} immediately establishes the following result.
\begin{lemma}\label{l:triv} 
If $m\geq 2$, then $K_{m-2}^{2\od(m)}p_{m-1}$ is absorbed by $p_m$.
\end{lemma}

The following straightforward lemma proves essential for subsequent proofs.
\begin{lemma}\label{l:red} 
For any $m\geq 1$ and $u\geq 0$, $p_u$ is absorbed by $p_m p_u$.
\end{lemma}
\begin{proof} 
Since the constant term of $p_m$ equals $1$, it follows that
  \[
    p_m p_u + p_u = (p_m + 1)p_u = p_m p_u.
  \]
\end{proof}

In contrast, proving the following lemmas requires involved algebraic manipulation.
\begin{lemma}\label{l:xp}
  \begin{enumerate}
   \item \label{l:xp0} If $m\geq 2$, then $(1+x)p_m = p_m + K_m/k^4$. 
   \item \label{l:xp1} If $m\geq 3$, then $p_m + x p_{m-1} = p_m$.
  \end{enumerate}
\end{lemma}
\begin{proof}
    \begin{enumerate}
    \item We proceed by induction on $m$. The base cases $m=2$ and $m=3$ follow from straightforward calculations. For $m\geq 4$, the inductive step proceeds as follows: If $m$ is even, the inductive definition of $p_m$, the induction hypothesis, and Eqs.~\eqref{eq:Ks} ensure that
\begin{align*}
&(1+x)p_m=\left(K_{m-1}p_{m-2}+\frac{K_{m-1}}{k^4}\right)\\
&\qquad\qquad\quad+\left(K_{m-1}K_{m-2}p_{m-2}+\frac{K_{m-1}K_{m-2}}{k^4}\right)+p_{m-1}+\frac{K_m}{k^4}\\
&= K_{m-1}p_{m-2}+K_{m-1}K_{m-2}p_{m-2}+p_{m-1}+\frac{K_m}{k^4}=p_m+\frac{K_m}{k^4},
\end{align*}   
   since the constant term of $p_{m-2}$ equals $1$. A similar argument applies when $m$ is odd.
 
\item Again, the proof proceeds by induction on $m$. The base case $m=3$ follows from straightforward calculations. Assume $m\geq 4$. If $m$ is even, the inductive definition of $p_m$ gives
  \begin{align*}
    p_m+&xp_{m-1}=(1+x)p_{m-1}+K_{m-1}(1+K_{m-2})p_{m-2},\\
                &= p_{m-1}+\left(\frac{K_{m-1}}{k^4}+K_{m-1}p_{m-2}\right)\\
                &\qquad\qquad\qquad\qquad +K_{m-1}K_{m-2}p_{m-2}, \quad \text{(by \Cref{l:xp} \Cref{l:xp0})}\\
    &= p_{m-1}+K_{m-1}p_{m-2}+K_{m-1}K_{m-2}p_{m-2}=p_m,
  \end{align*}
  since the constant term of $p_{m-1}$ equals $1$. Therefore, $p_m+x p_{m-1}= p_m$. If $m$ is odd, the recursive definitions of $p_m$ and $p_{m-1}$ yield
  \begin{align*}
    &p_{m}+xp_{m-1}=K_{m-2}^{2}p_{m-1}+K_{m-2}p_{m-2}
                  +(1+x)p_{m-2}\\&\qquad\qquad\qquad\qquad\qquad\qquad\qquad\qquad+xK_{m-2}(1+K_{m-3})p_{m-3}\\
    &= K_{m-2}^{2}p_{m-1}+K_{m-2}p_{m-2}
                   +p_{m-2}\\&\qquad+\frac{K_{m-2}}{k^4}+xK_{m-2}(1+K_{m-3})p_{m-3}. \quad \text{(by \Cref{l:xp} \Cref{l:xp0})}
  \end{align*} 
Since $K_{m-2}p_{m-2}+K_{m-2}/k^4= K_{m-2}p_{m-2}$ (because the constant term of $p_{m-2}$ equals $1$), it follows that
\begin{align*}
p_m+xp_{m-1}= K_{m-2}^2p_{m-1}+p_{m-2}+&K_{m-2}p_{m-2}\\&+xK_{m-2}(1+K_{m-3})p_{m-3},
\end{align*}
where $xK_{m-2}(1+K_{m-3})p_{m-3}$ is absorbed by $K_{m-2}p_{m-2}$ because
\begin{align*}
p_{m-2}+x(1+K_{m-3})p_{m-3}&= (p_{m-2}+xp_{m-3})+K_{m-4}^2p_{m-3}\\
&= p_{m-2}+K_{m-4}^2p_{m-3}= p_{m-2},
\end{align*}
by applying Eqs.~\eqref{eq:Ks} (which implies $xK_{m-3}=K_{m-4}^2$), the induction hypothesis and \Cref{l:triv}. Therefore, $p_m+xp_{m-1}= p_m$.
\end{enumerate}
\end{proof}
  
This lemma and the subsequent results in this section establish identities within the dioid $\DiZ$, serving as intermediate steps toward proving that the deformed Jacobsthal polynomials exhibit a form of idempotency.
\begin{lemma}\label{l:pKs}  
   If $m \geq 2$ is even, then
    \[
      p_m+K_mp_m+K_{m-1}^2p_m= p_m+K_mp_m+\frac{K_m^2}{k^4}.
    \] 
\end{lemma}
\begin{proof}
  From Eqs.~\eqref{eq:Ks}, $K_{m-1}^2=xK_m$. Thus, \Cref{l:xp} \Cref{l:xp0} leads to
  \begin{align*}
    p_m+K_mp_m+K_{m-1}^2p_m&=p_m+K_mp_m(1+x)= p_m+K_m\left(p_m+\frac{K_m}{k^4}\right).
  \end{align*}
\end{proof}

The following lemma provides key technical relations regarding the scaling behavior and index shifts of the parameters $K_m$.
\begin{lemma}\label{l:Kmk4}
  \begin{enumerate}
  \item\label{l:Kmk4i} If $m \geq 4$ is even, then
  \[
    p_m + K_{m-1} p_{m-1} = p_m + \frac{K_m}{k^4}.
  \]
  \item\label{l:Kmk4ii} If $m \geq 5$ is odd, then
  \[
    p_m + K_{m-2}^3 p_{m-2} = p_m + \frac{K_m}{k^4}.
  \]
  \end{enumerate}
\end{lemma}

\begin{proof}
  \begin{enumerate}
  \item 
  The base case $m=4$ follows from straightforward calculations. For $m\geq 6$, the recursive definition of $p_{m-1}$ allows us to write
  \begin{align*}
     &p_{m-2}+K_{m-2}p_{m-2}+p_{m-1} = \left(p_{m-2}+K_{m-2}p_{m-2}+K_{m-3}^2p_{m-2}\right)\\
                                   &\qquad +\left(p_{m-2}+p_{m-3}\right)+\left(p_{m-2}+K_{m-3}p_{m-3}\right)+K_{m-2}p_{m-2}\\
                                   &= \left(p_{m-2}+K_{m-2}p_{m-2}+\frac{K_{m-2}^2}{k^4}\right)+p_{m-2} \quad \text{(by \Cref{l:pKs,l:triv})}\\
                                   &\qquad +p_{m-2}+\frac{K_{m-2}}{k^4}+K_{m-2}p_{m-2} \quad \text{(by the induction hypothesis).}
  \end{align*}
Now, since the constant term of $p_{m-2}$ equals $1$, it follows that $K_{m-2}/k^4 \abs K_{m-2} p_{m-2}$. Therefore,
\[
   p_{m-2} + K_{m-2} p_{m-2} + p_{m-1} = p_{m-2} + K_{m-2} p_{m-2} + \frac{K_{m-2}^2}{k^4}.
\]
Multiplying both sides by $K_{m-1}$ we obtain
\begin{align*}
  K_{m-1} p_{m-2} + K_{m-1} &K_{m-2} p_{m-2} + K_{m-1} p_{m-1} \\
   &= K_{m-1} p_{m-2} + K_{m-1} K_{m-2} p_{m-2} + \frac{K_{m}}{k^4}.  
\end{align*}
Adding $p_{m-1}$ to both sides leads to the result.

\item The recursive definition of $p_m$ gives
  \begin{align*}
    p_m + K_{m-2}^3 p_{m-2}
      &= K_{m-2}^2 (p_{m-1} + K_{m-2} p_{m-2}) + (1 + K_{m-2}) p_{m-2} \\
      &= K_{m-2}^2 \left(p_{m-1} + \frac{K_{m-1}}{k^4}\right) + (1 + K_{m-2}) p_{m-2} \\
      &= p_m + \frac{K_m}{k^4},
  \end{align*}
  by \Cref{l:Kmk4} \Cref{l:Kmk4i} and the identity $K_m = K_{m-1} K_{m-2}^2$.
  \end{enumerate}
\end{proof}

The following lemma establishes some parity-dependent absorption relations between consecutive deformed Jacobsthal polynomials.

\begin{lemma}\label{l:pm-1Apm}
Let $m\geq 1$. If $m$ is odd, then $p_m$ absorbs $x(1+K_{m-1})p_{m-1}$; if $m$ is even, then $(1+K_{m-1})p_{m-1}$ absorbs $xp_m$.
\end{lemma}
\begin{proof}
The proof proceeds by induction on $m$. If $m=1$ or $m=2$, the result is trivial. Assume $m\geq 3$.

\textit{Case $m$ is odd:} The induction hypothesis ensures that $xp_{m-1}\abs (1+K_{m-2})p_{m-2}$. Thus, using the recursive definition of $p_m$, we find
\begin{align*}
p_m+x(1+K_{m-1})p_{m-1}&=K_{m-2}^2p_{m-1}+\big((1+K_{m-2})p_{m-2}+xp_{m-1}\big)\\
&\qquad\qquad\qquad\qquad\qquad\qquad\qquad\qquad+\underbrace{xK_{m-1}}_{K_{m-2}^2}p_{m-1}\\
&= K_{m-2}^2p_{m-1}+(1+K_{m-2})p_{m-2}=p_m.
\end{align*}

\textit{Case $m$ is even:} From the induction hypothesis, $xK_{m-1}(1+K_{m-2})p_{m-2}\abs K_{m-1}p_{m-1}$. Furthermore, the recursive definition of $p_m$ yields
\begin{align*}
xp_m+(1+K_{m-1})p_{m-1}&=xp_{m-1}+\big(xK_{m-1}(1+K_{m-2})p_{m-2}+K_{m-1}p_{m-1}\big)\\
&\qquad\qquad\qquad\qquad\qquad\qquad\qquad\qquad\qquad\qquad+p_{m-1}\\
&= xp_{m-1}+K_{m-1}p_{m-1}+p_{m-1}\\
&= p_{m-1}+\frac{K_{m-1}}{k^4}+K_{m-1}p_{m-1} \quad \text{(by \Cref{l:xp} \Cref{l:xp0})}\\
&= p_{m-1}+K_{m-1}p_{m-1}.
\end{align*}
\end{proof}

The following lemma establishes an additional absorption relation for odd indices, involving a scaled sum of consecutive polynomials.

\begin{lemma}\label{l:pm-1Km-1}
  If $m\geq 1$ is odd, then $p_{m-1}+K_{m-1}p_m$ absorbs  $p_m$.
\end{lemma}
\begin{proof}
If $m = 1$, the result is obvious. If $m \geq 3$, then from the recursive definition of $p_m$ and the relation $K_{m-2}^2 = x K_{m-1}$, we obtain
  \begin{align*}
    p_m + p_{m-1} + K_{m-1}& p_m
       = K_{m-1} \big(x p_{m-1} + (1 + K_{m-2}) p_{m-2}\big) + p_{m-1} + p_{m-2} \\
      &\qquad + K_{m-2} p_{m-2} + K_m p_{m-1} \\
      &= K_{m-1}(1 + K_{m-2}) p_{m-2} + p_{m-1} \\
      &\qquad + K_{m-2} p_{m-2} + K_m p_{m-1} \quad  {\text{(by \Cref{l:pm-1Apm,l:triv})}} \\
      &= K_{m-1} p_m + p_{m-1} + K_{m-2} p_{m-2} \\
      &= K_{m-1} p_m + p_{m-1} + \frac{K_{m-1}}{k^4} \quad {\text{(by \Cref{l:Kmk4} \Cref{l:Kmk4i})}} \\
      &= K_{m-1} p_m + p_{m-1}.
  \end{align*}
\end{proof}

The next result provides a simplification for mixed products of two deformed Jacobsthal polynomials with consecutive indices.

\begin{lemma}\label{l:UNewA}
\begin{enumerate}
\item\label{l:UNewA1} If $m\geq 2$ is even and $p_{m-1}p_{m-2}\abs p_{m-2}+K_{m-2}p_{m-1}$, then 
\[
K_{m-1}(1+K_{m-2})p_{m-1}p_{m-2}\abs p_m+K_m p_{m-1}.
\]
\item\label{l:UNewA2} If $m\geq 2$ is odd and $p_{m-1}p_{m-2}\abs p_{m-1}+K_{m-1}p_{m-2}$, then
  \[
    (1+K_{m-2})p_{m-1}p_{m-2}\abs p_{m-1}+K_{m-1}p_m.
  \]
\end{enumerate}
\end{lemma}

\begin{proof}
\begin{enumerate}
  \item By the second hypothesis, 
  \begin{align*}
    (1+K_{m-2})p_{m-1}p_{m-2}
      &\abs (1+K_{m-2})p_{m-2} + (1+K_{m-2})K_{m-2}p_{m-1} \\
      &= p_{m-2} + K_{m-2}\big(p_{m-1} + p_{m-2} + K_{m-2}p_{m-1}\big) \\    
      &= p_{m-2} + K_{m-2}p_{m-2} + K_{m-2}^2p_{m-1},
  \end{align*}
  where the last step follows from \Cref{l:pm-1Km-1}. Multiplying  by $K_{m-1}$ leads to 
  \[
    K_{m-1}(1+K_{m-2})p_{m-1}p_{m-2} \abs K_{m-1}(1+K_{m-2})p_{m-2} + K_m p_{m-1}, 
  \]
  which, in turn, implies that
  \begin{align*}
    K_{m-1}(1+K_{m-2})p_{m-1}p_{m-2}
      &\abs p_{m-1} + K_{m-1}(1+K_{m-2})p_{m-2} \\
      &\qquad \qquad \qquad \qquad\qquad\quad+ K_m p_{m-1} \\
      &= p_m + K_m p_{m-1}.
  \end{align*}

\item From $K_{m-2}^2 = x K_{m-1}$ and \Cref{l:xp} \Cref{l:xp1}, it follows that $K_{m-2}^2 p_{m-3} \abs K_{m-1} p_{m-2}$, which shows
  \begin{equation}\label{eq:domA}
    K_{m-2}^2 p_{m-3} \abs p_{m-1} + K_{m-1} p_{m-2}.
  \end{equation}
  On the other hand, \Cref{l:triv} implies that $x K_{m-3} p_{m-3} \abs p_{m-2}$. Since $K_{m-2}^2 = x K_{m-1}$, it follows that $K_{m-2}^2 K_{m-3} p_{m-3} \abs K_{m-1} p_{m-2}$, and thus
  \[
    K_{m-2}^2 K_{m-3} p_{m-3} \abs p_{m-1} + K_{m-1} p_{m-2}.
  \]
  Combining this with \Cref{eq:domA} yields
  \begin{equation}\label{eq:domB}
    K_{m-2}^2(1 + K_{m-3}) p_{m-3} \abs p_{m-1} + K_{m-1} p_{m-2}.
  \end{equation}
Additionally, \Cref{l:Kmk4} \Cref{l:Kmk4i} yields
  \begin{equation}\label{eq:domC}
    K_{m-2}p_{m-2} \abs p_{m-1} + K_{m-1}p_{m-2}.
  \end{equation}
  Thus, combining the recursive definition of $p_{m-1}$ with \Cref{eq:domB,eq:domC} results in
  \begin{equation}\label{eq:UNewB}
    K_{m-2}p_{m-1} \abs p_{m-1} + K_{m-1}(1 + K_{m-2})p_{m-2}.
  \end{equation}
  Now, using the recursive definition of $p_m$ along with \Cref{eq:UNewB}, it follows that
  \begin{align}
    p&_{m-1} + K_{m-1}p_m 
      = \big(p_{m-1} + K_{m-1}(1 + K_{m-2})p_{m-2}\big) + K_m p_{m-1} \nonumber \\
      &= \big(K_{m-2}p_{m-1} + p_{m-1} + K_{m-1}(1 + K_{m-2})p_{m-2}\big) + K_m p_{m-1}. \label{eq:UNewC}
  \end{align}
  Recall the hypothesis that $p_{m-1}p_{m-2} \abs p_{m-1} + K_{m-1}p_{m-2}$. Multiplying both sides by $(1 + K_{m-2})$ leads to
  \begin{equation}\label{eq:mixA}
\begin{split}
  (1 + K_{m-2})p_{m-1}p_{m-2} &\abs p_{m-1} + K_{m-2}p_{m-1}\\ 
    &\quad + K_{m-1}(1 + K_{m-2})p_{m-2}.
\end{split}
\end{equation}
Therefore,
\begin{align*}
 \qquad\quad  & (1+K_{m-2})p_{m-2}p_{m-1} + p_{m-1} + K_{m-1}p_m \\
    &\quad = (1+K_{m-2})p_{m-2}p_{m-1} + K_{m}p_{m-1} \\
    &\qquad + \big(K_{m-2}p_{m-1} + p_{m-1} + K_{m-1}(1+K_{m-2})p_{m-2}\big) &{\text{(by \Cref{eq:UNewC})}} \\
    &\quad = \big(K_{m-2}p_{m-1} + p_{m-1} + K_{m-1}(1+K_{m-2})p_{m-2}\big) \\
    &\qquad \qquad\qquad \qquad\qquad \qquad\qquad \qquad \qquad \qquad+ K_m p_{m-1} &{\text{(by \Cref{eq:mixA})}} \\
    &\quad = p_{m-1} + K_{m-1}p_{m}, &{\text{(by \Cref{eq:UNewC})}} 
\end{align*}
where the last step follows from the recursive definition of $p_m$.
\end{enumerate}
\end{proof}

The next lemma provides a refinement of the previous result on simplifying mixed products of deformed Jacobsthal polynomials. In particular, Condition \eqref{eq:condSqrt} is introduced to establish an idempotency property for these polynomials by induction in \Cref{thm:idem}.

\begin{lemma}\label{l:UNewTa} Let $m\geq 1$. If 
\begin{equation}\label{eq:condSqrt}
p_{j}^2+K_{j}p_{j}= p_{j}+K_{j}p_{j},\text{ for all $0\leq j\leq m-1$,}
\end{equation}
then
  \[
p_mp_{m-1}\abs K_{m-1}^{\od(m)}p_m+K_m^{\ev(m)}p_{m-1}.
  \]
\end{lemma}
\begin{proof}
The proof proceeds by induction on $m$. If $m=1$ or $m=2$, the result is trivial. Assume that $m > 2$ and that Condition~\eqref{eq:condSqrt} holds.

\textit{Case $m$ is odd:} The recursive definition of $p_m$ ensures that
\begin{align}
  p_m p_{m-1} &+ p_{m-1} + K_{m-1}p_m\nonumber\\
    &= K_{m-2}^2p_{m-1}^2 + (1+K_{m-2})p_{m-1}p_{m-2} + p_{m-1} + K_{m-1}p_m \nonumber \\
    &= K_{m-2}^2p_{m-1}^2 + p_{m-1} + K_{m-1}p_m, \label{eq:UNewD}
\end{align}
by \Cref{l:UNewA} \Cref{l:UNewA2}. Observe that $K_{m-2}^2p_{m-1}\abs p_m$. Multiplying both sides of this relation by $K_{m-1}$, we find that  $K_{m}p_{m-1}\abs K_{m-1}p_m$, which implies $K_{m}p_{m-1}\abs p_{m-1}+K_{m-1}p_m$. Then, from~\Cref{eq:UNewD}, it follows that
\begin{align*}
  p_{m}p_{m-1}+p_{m-1}+K_{m-1}p_m
    &= K_{m-2}^2(p_{m-1}^2+K_{m-1}p_{m-1})+p_{m-1}+K_{m-1}p_m\\
    &= K_{m-2}^2(p_{m-1}+K_{m-1}p_{m-1})+p_{m-1}+K_{m-1}p_m\\
    &= K_{m-2}^2p_{m-1}+p_{m-1}+K_{m-1}p_m.
\end{align*}
It remains to show that $K_{m-2}^2 p_{m-1}$ is absorbed by the last expression. Indeed, \Cref{l:xp} \Cref{l:xp1} implies that $xp_{m-1} \abs p_m$. Multiplying both sides of this relation by $K_{m-1}$, we get $K_{m-2}^2 p_{m-1} \abs K_{m-1} p_m$ (see \Cref{eq:Ks}).

\textit{Case $m$ is even:} The recursive definition of $p_m$ gives
\begin{align*}
  p_{m-1}p_m + p_m + K_m p_{m-1} 
    &= p_{m-1}^2 + K_{m-1}(1 + K_{m-2})p_{m-2}p_{m-1}  + p_m + K_m p_{m-1} \\
    &= p_{m-1}^2 + p_m + K_m p_{m-1},
\end{align*}
by \Cref{l:UNewA} \Cref{l:UNewA1} and \Cref{eq:condSqrt}. By \Cref{l:Kmk4} \Cref{l:Kmk4i}, we have $K_{m-1}p_{m-1} \abs p_m + K_m p_{m-1}$. Thus,
\begin{align*}
  p_{m-1}p_m+p_m+K_mp_{m-1}
    &= p_{m-1}^2+K_{m-1}p_{m-1}+ p_m+K_mp_{m-1}\\
    &= p_{m-1}+K_{m-1}p_{m-1}+p_m+K_mp_{m-1}\\
    &= p_m+K_mp_{m-1}. \quad {\text{(by \Cref{l:triv})}}
\end{align*}
\end{proof}

The following lemma establishes the key absorption estimates required for the inductive step in the proof of \Cref{thm:idem}.

\begin{lemma}\label{l:pAbs}
  Let $m \geq 2$. If $p_j^2 + K_j p_j = p_j + K_j p_j$ for all $0 \leq j \leq m-2$, then the following hold:
  \begin{enumerate}
    \item\label{l:pAbsi} If $m$ is even, then
      \[
        K_{m-1}(1 + K_{m-2})p_{m-1}p_{m-2} \abs (1 + K_m)p_m.
      \]
    \item\label{l:pAbsii} If $m$ is odd, then
      \[
        K_{m-2}^2(1 + K_{m-2})p_{m-1}p_{m-2} \abs (1 + K_m)p_m.
      \]
  \end{enumerate}
\end{lemma}

\begin{proof}
  \begin{enumerate}
\item If $m=2$, then $p_{m-2}=0$ and the result is trivial. Thus, assume $m \geq 4$. By \Cref{l:triv},
    \begin{equation}
      \label{eq:UNewH}
      K_m p_{m-1} \abs K_m p_m \abs p_m + K_m p_m.
    \end{equation}
    Moreover, since $K_{m-1}K_{m-2}p_{m-2}$ is a summand in the recursive definition of $p_m$, it follows that
    \begin{equation}\label{eq:UNewJ}
      K_{m-1}K_{m-2}p_{m-2} \abs p_m \abs p_m + K_m p_m. 
    \end{equation}
    Hence, \Cref{eq:UNewH,eq:UNewJ} imply that
    \begin{equation}
      \label{eq:UNewK}
      K_m p_{m-1} + K_{m-1}K_{m-2}p_{m-2} \abs (1 + K_m)p_m. 
    \end{equation}
    Applying \Cref{l:UNewTa} to the odd $m-1$ and multiplying by $K_{m-1}K_{m-2}$ gives
    \[
      K_{m-1}K_{m-2}p_{m-1}p_{m-2} \abs K_m p_{m-1} + K_{m-1}K_{m-2}p_{m-2},
    \]
    which, together with~\Cref{eq:UNewK} and the transitivity of absorption, gives
    \begin{equation}
      \label{eq:UNewM}
      K_{m-1}K_{m-2}p_{m-1}p_{m-2} \abs (1 + K_m)p_m.
    \end{equation}
    Next, applying \Cref{l:red,l:UNewTa} again yields
    \begin{align*}
      p_{m-1} &\abs p_{m-1}p_{m-2} \abs K_{m-2}p_{m-1} + p_{m-2} \\
              &\abs K_{m-2}p_{m-1} + p_{m-2} + K_{m-1}K_{m-2}(1 + K_{m-2})p_{m-2} \\
              &= p_{m-2} + K_{m-2}p_m.
    \end{align*}

That is, $p_{m-1} \abs p_{m-2} + K_{m-2}p_m$. Multiplying both sides of this relation by $K_{m-1}K_{m-2}$, we find that
    \begin{align*}
      \qquad K_{m-1}K_{m-2}p_{m-1} 
        &\abs K_{m-1}K_{m-2}p_{m-2} + K_m p_m \\
        &\abs p_{m-1} + K_{m-1}p_{m-2} + K_{m-1}K_{m-2}p_{m-2} + K_m p_m \\
        &= p_m + K_m p_m. 
    \end{align*}
    Thus,
    \begin{equation}
      \label{eq:UNewL}
      K_{m-1}K_{m-2}p_{m-1} \abs (1 + K_m)p_m.
    \end{equation} 
    Additionally, the recursive definition of $p_m$ implies $K_{m-1}p_{m-2} \abs p_m$, while $p_m \abs p_m + K_m p_m$ holds trivially. By transitivity, $K_{m-1}p_{m-2} \abs (1 + K_m)p_m$, which together with~\Cref{eq:UNewL} yields
    \[
      K_{m-1}K_{m-2}p_{m-1} + K_{m-1}p_{m-2} \abs (1 + K_m)p_m.
    \]
    Combining this with $K_{m-1}p_{m-1}p_{m-2} \abs K_{m-1}K_{m-2}p_{m-1} + K_{m-1}p_{m-2}$ (by \Cref{l:UNewTa}) via transitivity gives
    \begin{equation}
      \label{eq:UNewN}
      K_{m-1}p_{m-1}p_{m-2} \abs (1 + K_m)p_m.  
    \end{equation}
    Finally, the absorptions in~\Cref{eq:UNewM,eq:UNewN} lead to
    \[
      K_{m-1}(1 + K_{m-2})p_{m-1}p_{m-2} \abs (1 + K_m)p_m.
    \]
    
    \item Straightforward calculations show that the result holds for $m=3$. Thus, assume $m\geq 5$. Now, the recursive definition of $p_m$ contains $p_{m-2}$ as a summand, so $K_m p_{m-2} \abs K_m p_m$. It follows that $K_m p_{m-2} \abs p_m + K_m p_m$. Moreover, $K_{m-2}^2 p_{m-1}$ is also a summand of $p_m$, so $K_{m-2}^2 p_{m-1} \abs p_m + K_m p_m$. Hence,
  \begin{equation}
    \label{eq:UNewO}
    K_{m-2}^2 p_{m-1} + K_m p_{m-2} \abs (1 + K_m) p_m.
  \end{equation}
  In turn, applying \Cref{l:UNewTa} to the even integer $m-1$ and multiplying by $K_{m-2}^2$ we obtain
  \begin{equation}
    \label{eq:UNewP}
    K_{m-2}^2 p_{m-1} p_{m-2} \abs K_{m-2}^2 p_{m-1} + K_m p_{m-2}.  
  \end{equation}
  Combining \Cref{eq:UNewP,eq:UNewO} by transitivity yields
  \begin{equation}
    \label{eq:UNewQ}
    K_{m-2}^2 p_{m-1} p_{m-2} \abs (1 + K_m) p_m.
  \end{equation}    
 Additionally, by \Cref{l:pm-1Apm}, $x(1+K_{m-3})p_{m-3}\abs p_{m-2}$. Multiplying this relation by $K_{m-1}$ gives
\begin{equation}
  \label{eq:UNewA1}
  K_{m-2}^2(1+K_{m-3})p_{m-3}\abs K_{m-1}p_{m-2}\abs p_{m-1}+K_{m-1}p_{m-2}.
\end{equation}
Observe also that by \Cref{l:Kmk4} \Cref{l:Kmk4i}, $K_{m-2}p_{m-2}\abs p_{m-1}+K_{m-1}/k^4$, whence
\begin{equation}
  \label{eq:UNewA2}
  K_{m-2}p_{m-2}\abs p_{m-1}+\frac{K_{m-1}}{k^4}+K_{m-1}p_{m-2}= p_{m-1}+K_{m-1}p_{m-2}.
\end{equation}
In view of \Cref{eq:UNewA1,eq:UNewA2}, it follows that
\begin{align*}
  K_{m-2}p_{m-1} 
    &= K_{m-2}p_{m-2} + K_{m-2}^2(1+K_{m-3})p_{m-3}\\
    &\abs p_{m-1}+K_{m-1}p_{m-2},
\end{align*}
where
\begin{align*}
  p_{m-1}+K_{m-1}p_{m-2}
    &\abs p_{m-1}+K_{m-1}p_{m-2}+K_{m-1}K_{m-2}p_{m-2}\\
    &\qquad +K_mp_{m-1}=p_{m-1}+K_{m-1}p_m.
\end{align*}
By transitivity, $K_{m-2}p_{m-1} \abs p_{m-1} + K_{m-1}p_m$. Multiplying both sides of this relation by $K_{m-2}^2$ implies
\begin{equation}\label{eq:ExA}
  K_{m-2}^3 p_{m-1} \abs K_{m-2}^2 p_{m-1} + K_m p_m \abs p_m + K_m p_m.
\end{equation}
Furthermore,
\begin{equation}\label{eq:ExB}
  K_{m-2}K_m p_{m-2} \abs K_m p_m \abs p_m + K_m p_m.
\end{equation}
Therefore, from \Cref{eq:ExA,eq:ExB}, it follows that
\begin{equation}\label{eq:ExC}
  K_{m-2}^3 p_{m-1} + K_{m-2}K_m p_{m-2} \abs (1 + K_m)p_m.
\end{equation}
Applying \Cref{l:UNewTa} to the even integer $m-1$ and multiplying by $K_{m-2}^3$, we find that
\[
  K_{m-2}^3 p_{m-1} p_{m-2} \abs K_{m-2}^3 p_{m-1} + K_{m-2}K_m p_{m-2},
\]
which, by transitivity with \Cref{eq:ExC}, implies
\begin{equation}
  \label{eq:UNewA3}
  K_{m-2}^3 p_{m-1} p_{m-2} \abs (1 + K_m)p_m.  
\end{equation}
Combining \Cref{eq:UNewQ,eq:UNewA3} gives
\[
  K_{m-2}^2 p_{m-1} p_{m-2} + K_{m-2}^3 p_{m-1} p_{m-2} \abs (1 + K_m)p_m.
\]
\end{enumerate}
\end{proof}

The deformed Jacobsthal polynomials satisfy an idempotency property.

\begin{theorem}\label{thm:idem}
  For any $m \geq 0$,
  \[
    p_m^2 + K_m p_m = p_m + K_m p_m.
  \]  
\end{theorem}

\begin{proof}
  The proof proceeds by induction on $m$. The base cases $m=0,1$ follow because $p_m^2=p_m$ for these values. Assume now that $m\geq 2$. If $m$ is even, then the recursive definition of $p_m$ yields
  \begin{multline}
    p_m^2 = p_{m-1}^2 + K_{m-1}(1+K_{m-2})p_{m-1}p_{m-2} \\
    + K_{m-1}^2 K_{m-2}^2 p_{m-2}^2 + K_{m-1}^2 K_{m-2} p_{m-2}^2 + K_{m-1}^2 p_{m-2}^2, \label{eq:pm2ev}
  \end{multline}
  where nearly every squared term $p_{m-j}^2$ ($j=1,2$) in this expression can be simplified by the induction hypothesis. For instance, the first summand $p_{m-1}^2$ reduces to $p_{m-1}$ because $K_{m-1}p_{m-1}$ can be added to the right-hand side of \Cref{eq:pm2ev}, as it is absorbed by $K_{m-1}p_{m-1}p_{m-2}$ according to \Cref{l:red}. Similarly, $K_{m-1}^2K_{m-2}p_{m-2}^2$ reduces to $K_{m-1}^2K_{m-2}p_{m-2}$ because the term $K_{m-1}^2K_{m-2}^2p_{m-2}$ can be added, which is absorbed by $K_{m-1}^2K_{m-2}^2p_{m-2}^2$ again by \Cref{l:red}. The term $K_{m-1}^2p_{m-2}^2$ is handled in the same manner. Hence, it follows that
\begin{multline*}
  p_m^2 = p_{m-1} + K_{m-1}(1+K_{m-2})p_{m-1}p_{m-2} \\
  + K_{m-1}^2 K_{m-2}^2 p_{m-2}^2 + K_{m-1}^2 K_{m-2} p_{m-2} + K_{m-1}^2 p_{m-2}.
\end{multline*}
Now, from $p_{m-2} \abs p_{m-1}p_{m-2}$ and \Cref{l:pAbs}, it follows that
\begin{multline}\label{eq:pm2evB}
  p_m^2 + K_m p_m = p_m + K_m p_m + K_{m-1}^2 K_{m-2}^2 p_{m-2}^2 \\
  + K_{m-1}^2 K_{m-2} p_{m-2} + K_{m-1}^2 p_{m-2}.
\end{multline}
To obtain the result, it suffices to prove that the last three summands on the right-hand side of \Cref{eq:pm2evB} are absorbed by $K_m p_m$. Indeed, the recursive definition of $p_m$ implies $K_{m-1}K_{m-2} \abs p_m$, which, upon multiplying by $K_m$, yields that  $K_m p_m$ absorbs $K_{m-1}^2 K_{m-2}^3 p_{m-2}$. Thus, the term $K_{m-1}^2 K_{m-2}^3 p_{m-2}$ can be added to the right-hand side of \Cref{eq:pm2evB}. By applying the induction hypothesis, $K_{m-1}^2 K_{m-2}^2 p_{m-2}^2$ then reduces to $K_{m-1}^2 K_{m-2}^2 p_{m-2}$. Furthermore, $K_{m-1}p_{m-2} \abs p_m$, which yields $K_{m-1}^2 K_{m-2}^2 p_{m-2} \abs K_m p_m$ upon multiplying by $K_m$. Next, the recursive definition of $p_{m-1}$ shows that $K_{m-1}p_{m-2}$ is a summand of $K_{m-2}p_{m-1}$, so $K_{m-1}p_{m-2} \abs K_{m-2}p_{m-1}$. Multiplying by $K_{m-1}K_{m-2}$ gives $K_{m-1}^2 K_{m-2} p_{m-2} \abs K_m p_{m-1}$. Since $K_m p_{m-1} \abs K_m p_m$ from the recursive definition of $p_m$, transitivity ensures that $K_{m-1}^2 K_{m-2} p_{m-2} \abs K_m p_m$.
Finally, \Cref{l:xp} \Cref{l:xp1} and \Cref{l:triv} yield
\[
  x p_{m-2} \abs p_{m-1} \abs p_m,
\]    
which, upon multiplying by $K_m$, gives $K_{m-1}^2 p_{m-2} = x K_m p_{m-2} \abs K_m p_m$.

The case where $m$ is odd is similar.
\end{proof}

Applying \Cref{thm:idem} iteratively produces the following identity for weighted sums of $p_m$.
\begin{corollary}\label{c:prov}
  For all $\ell, m \geq 1$,
  \[
    \sum_{j=0}^{\ell-1} K_m^j p_m^2 + K_m^\ell p_m = \sum_{j=0}^\ell K_m^j p_m.
  \]
\end{corollary}
\begin{proof}
  This follows by a routine induction argument on $\ell$, with \Cref{thm:idem} serving as the base case.
\end{proof}

Building on the previous result, the following corollary shows how to completely reduce higher-degree terms of $p_m$.

\begin{corollary}\label{c:NoPow}
  If $\ell\geq 0$, then for any $m\geq 1$,
  \[
    \sum_{j=0}^{\ell}K_m^jp_m^{\ell+1-j}= \sum_{j=0}^\ell K_m^j p_m.
  \]
\end{corollary}
\begin{proof}
  The proof proceeds by induction on $\ell$. The base case $\ell=0$ is the trivial $p_m= p_m$. Furthermore,
  \begin{align*}
    \sum_{j=0}^{\ell+1}K_m^jp_m^{\ell+2-j}
      &= p_m\sum_{j=0}^{\ell}K_m^jp_m^{\ell+1-j}+K_m^{\ell+1}p_m\\
      &= \sum_{j=0}^\ell K_m^jp_m^2+K_m^{\ell+1}p_m & {\text{(by the induction hypothesis)}}\\
      &= \sum_{j=0}^{\ell+1}K_m^jp_m &{\text{(by \Cref{c:prov}).}}
  \end{align*}
\end{proof}

The following two-step reduction formula facilitates the remaining estimates.
\begin{lemma}\label{l:Alst}
  Let $m \geq 4$. Then
  \begin{equation}
    \label{eq:recA}
    p_m + \frac{K_m}{k^4} = x^{\od(m)} \left( \sum_{j=0}^1 K_{m-1}^j p_{m-1} \right) + x^{\ev(m)} \left( \sum_{j=0}^3 K_{m-2}^j p_{m-2} \right).
  \end{equation}
\end{lemma}

\begin{proof}
  If $m$ is odd, the right-hand side of \Cref{eq:recA} becomes
  \begin{align*}
    x p_{m-1} &+ x K_{m-1} p_{m-1} + p_{m-2} + K_{m-2} p_{m-2} + K_{m-2}^2 p_{m-2} + K_{m-2}^3 p_{m-2} \\
    &= x p_{m-1} + \underbrace{K_{m-2}^2 p_{m-1} + p_{m-2} + K_{m-2} p_{m-2}}_{p_m} + K_{m-2}^3 p_{m-2},
  \end{align*}
  since $x K_{m-1} = K_{m-2}^2$, and $K_{m-2}^2 p_{m-2} \abs K_{m-2}^2 p_{m-1}$ by \Cref{l:triv}. Next, applying \Cref{l:xp} \Cref{l:xp1} shows that the right-hand side has exact asymptotic order $p_m + K_{m-2}^3 p_{m-2}$. The result then follows from \Cref{l:Kmk4} \Cref{l:Kmk4ii}. If $m$ is even, substituting $xK_{m-2} = K_{m-3}^2$, $xK_{m-2}^2 = K_{m-1}$, and $xK_{m-2}^3 = K_{m-1}K_{m-2}$ into the right-hand side of \Cref{eq:recA} yields
  \begin{align*}
    & (p_{m-1} + xp_{m-2}) + K_{m-1}p_{m-1} + K_{m-3}^2 p_{m-2} + K_{m-1}p_{m-2} + K_{m-1}K_{m-2}p_{m-2} \\
    &= \underbrace{p_{m-1} + K_{m-1}p_{m-2} + K_{m-1}K_{m-2}p_{m-2}}_{p_m} + K_{m-1}p_{m-1} + (p_{m-1} + K_{m-3}^2 p_{m-2}) \\
    &= p_m + \frac{K_m}{k^4} + p_{m-1} \quad \text{(by \Cref{l:Kmk4} \Cref{l:Kmk4i} and \Cref{l:triv})} \\
    &= p_m + \frac{K_m}{k^4} \quad \text{(by \Cref{l:triv}).}
  \end{align*}
\end{proof}

\section{A geometric series versus deformed Jacobsthal polynomials}\label{s:geo}
In addition to the deformed Jacobsthal polynomials, the geometric series $S = \sum_{j=0}^\infty k^{2j} x^j$ must also be considered because it is closely related to the initial conditions of recurrence \eqref{eq:Aro} (see \Cref{l:2}); consequently, as we show in \Cref{thm:main}, its monomials appear in the tail of the \emph{reflected} polynomial of $t_{m,n}$ when $m \geq 4$, $n$ is odd, and $n \to \infty$. To pave the way for this result, this section investigates the algebraic behavior of $S$ when combined with the deformed Jacobsthal polynomials within the dioid $\DiZ$. First, observe that $S$ is idempotent in $\DiZ$:
\begin{lemma}\label{l:SIdem}
  \[
    S^2 = S.
  \]
\end{lemma}
\begin{proof}
  Since $S = \frac{1}{1-k^2x}$, formal differentiation of both sides with respect to $x$ yields
  \[
    S^2 = \frac{1}{(1-k^2x)^2} = 1 + 2(k^2x) + 3(k^2x)^2 + \cdots = S.
  \]
\end{proof}

To analyze the interplay between the geometric series $S$ and the deformed Jacobsthal polynomials, we establish the following absorption bounds.

\begin{lemma}\label{l:3}
  \begin{enumerate}
    \item\label{l:3i} If $m \geq 1$, the partial sum $\sum_{j=0}^{J_m-1} k^{2j-2} x^j$ is absorbed by $p_m$.
    \item\label{l:3ii} Let $m \geq 3$. If $m$ is odd, then the sum
      \[
        \sum_{u=J_{m-2}}^{2J_{m-2}-1} k^{2u-2+2J_{m-3}} x^u + \sum_{u=2J_{m-2}}^{J_m-1} k^{2u-2+4J_{m-3}} x^u
      \]
      is absorbed by $p_m$. On the other hand, if $m$ is even, the sum
      \begin{equation}\label{eq:inter2ev}
        \sum_{u=J_{m-1}}^{2^{m-2}-1}k^{2u-2+2J_{m-2}}x^u+\sum_{u=2^{m-2}}^{2^{m-2}+J_{m-2}-1} k^{2u-2+2^{m-2}} x^u
      \end{equation}
      is also absorbed by $p_m$.
  \end{enumerate}
\end{lemma}
\begin{proof}
  \begin{enumerate}
 \item We proceed by induction on $m$. For the base cases $m=1,2$, direct evaluation gives us
    \[
      p_m + \sum_{j=0}^{J_m-1} k^{2j-2} x^j = 1 + \frac{1}{k^2} = 1= p_m.
    \]
    Assume now that $m \geq 2$. For the inductive step, recall that $J_{n+1} + J_n = 2^n$ holds for any non-negative integer $n$, by \Cref{eq:cldJ}. If $m$ is even, then from the recursive definition of $p_{m+1}$ and the induction hypothesis, it follows that
    \begin{align*}
       & p_{m+1} + \sum_{j=0}^{J_{m+1}-1} k^{2j-2} x^j =\left( p_{m-1} + \sum_{j=0}^{J_{m-1}-1} k^{2j-2} x^j \right) \\
        &\quad + \left( K_{m-1} p_{m-1} + \sum_{j=J_{m-1}}^{2J_{m-1}-1} k^{2j-2} x^j \right)  + \left( K_{m-1}^2 p_m + \sum_{j=2J_{m-1}}^{J_{m+1}-1} k^{2j-2} x^j \right) \\
        &=  p_{m-1} + K_{m-1} \left( p_{m-1} + \sum_{u=0}^{J_{m-1}-1} k^{2u-2-2J_{m-2}} x^u \right) \\
        &\quad + K_{m-1}^2 \left( p_m + \sum_{u=0}^{J_m-1} k^{2u-2-4J_{m-2}} x^u \right) \\
        &=  p_{m-1} + K_{m-1} \biggl(  p_{m-1} + \sum_{j=0}^{J_{m-1}-1} k^{2j-2} x^j 
          + \sum_{u=0}^{J_{m-1}-1} k^{2u-2-2J_{m-2}} x^u \biggr) \\
        &\qquad\qquad\qquad\qquad + K_{m-1}^2 \biggl( p_m + \sum_{j=0}^{J_m-1} k^{2j-2} x^j 
          + \sum_{u=0}^{J_m-1} k^{2u-4J_{m-2}-2} x^u \biggr) \\
        &=  p_{m-1} + K_{m-1} \left(  p_{m-1} + \sum_{j=0}^{J_{m-1}-1} k^{2j-2} x^j \right)  \\
        &\qquad\qquad\qquad\qquad\qquad\qquad\qquad\qquad+ K_{m-1}^2 \left(  p_m + \sum_{j=0}^{J_m-1} k^{2j-2} x^j \right) \\
        &=  p_{m-1} + K_{m-1} p_{m-1} + K_{m-1}^2 p_m = p_{m+1}.
    \end{align*}
Similarly, if $m$ is odd, the induction hypothesis, together with the identities $J_{n+1}+J_n=2^n$ and $J_{n+2}-J_n=2^n$ for all $n\geq 0$, implies that
\begin{align*}
  \qquad &p_{m+1} + \sum_{j=0}^{J_{m+1}-1} k^{2j-2} x^j 
    = \left( p_m + \sum_{j=0}^{J_m-1} k^{2j-2} x^j \right) \\ 
    &\quad\qquad\qquad + K_m \left( p_{m-1} + \sum_{j=0}^{J_{m-1}-1} k^{2j-2} x^j + \sum_{u=0}^{J_{m-1}-1} k^{2u-2-2J_{m-2}} x^u \right) \\
    &\quad\qquad + K_m K_{m-1} \left( p_{m-1} + \sum_{j=0}^{J_{m-1}-1} k^{2j-2} x^j + \sum_{u=0}^{J_{m-1}-1} k^{2u-2-2^{m-1}} x^u \right) \\
    &= p_m + K_m p_{m-1} + K_m K_{m-1} p_{m-1} = p_{m+1}.
\end{align*}

\item Assume that $m$ is odd. Then, the recursive definition of $p_m$ leads to
    \begin{align*}
     \qquad\quad p_m &+ \sum_{u=J_{m-2}}^{2J_{m-2}-1} k^{2u-2+2J_{m-3}} x^u + \sum_{u=2J_{m-2}}^{J_m-1} k^{2u-2+4J_{m-3}} x^u \\
      &= p_{m-2} + K_{m-2} \left( p_{m-2} + \sum_{u=J_{m-2}}^{2J_{m-2}-1} k^{2u-2+2J_{m-3}-2^{m-2}} x^{u-J_{m-2}} \right) \\
      &\quad\qquad\qquad + K_{m-2}^2 \left( p_{m-1} + \sum_{u=2J_{m-2}}^{J_m-1} k^{2u-2+4J_{m-3}-2^{m-1}} x^{u-2J_{m-2}} \right) \\
      &= p_{m-2} + K_{m-2} \left( p_{m-2} + \sum_{j=0}^{J_{m-2}-1} k^{2j-2} x^j \right) \\
      &\quad \quad\qquad\qquad+ K_{m-2}^2 \left( p_{m-1} + \sum_{j=0}^{J_{m-1}-1} k^{2j-2} x^j \right) \\
      &= p_{m-2} +  K_{m-2} p_{m-2} +  K_{m-2}^2 p_{m-1} \quad \text{(by \Cref{l:3} \Cref{l:3i})} \\
      &=  p_m.
    \end{align*}
    In the even case, the recursive definition of $p_m$ implies that $p_m$ absorbs $K_{m-1}K_{m-2}p_{m-2}$ and $K_{m-1}p_{m-2}$. To obtain the result, we only need to prove that each sum in \Cref{eq:inter2ev} is absorbed by $p_m$, according to \Cref{p:Aabs} \Cref{p:Aabs3}. Indeed,
  \begin{align*}
   \qquad p_m &+ \sum_{u=2^{m-2}}^{2^{m-2}+J_{m-2}-1} k^{2u-2+2^{m-2}} x^u \\
        &= p_m + K_{m-1}K_{m-2} \left( p_{m-2} + \sum_{u=2^{m-2}}^{2^{m-2}+J_{m-2}-1} k^{2u-2-2^{m-1}} x^{u-2^{m-2}} \right) \\
        &= p_m + K_{m-1}K_{m-2} \left( p_{m-2} + \sum_{j=0}^{J_{m-2}-1} k^{2j-2} x^j \right) \\
        &= p_m + K_{m-1}K_{m-2} p_{m-2} \quad \text{(by \Cref{l:3} \Cref{l:3i})} \\
        &= p_m,
  \end{align*}
  and 
  \begin{align*}
  	\qquad p_m+\sum_{u=J_{m-1}}^{2^{m-2}-1}k^{2u-2+2J_{m-2}}x^u
	& = p_m+K_{m-1}\left(p_{m-2}+\sum_{j=0}^{J_{m-2}-1}k^{2j-2}x^j\right)\\
	&= p_m +K_{m-1}p_{m-2}= p_m.
  \end{align*}
    \end{enumerate}
\end{proof}

We establish the following lemma to describe how powers of the polynomial $p_m$ interact with the geometric series $S$ under exact asymptotic order.
\begin{lemma}\label{l:4}
  For every $m \geq 2$, the following statements hold:
  \begin{enumerate}
    \item\label{l:4i}
      \[
        p_m S = p_m + \frac{K_m}{k^2} S.
      \]
    \item\label{l:4ii} For any $\ell \geq 1$,
      \[
        p_m^\ell S = \sum_{u=0}^{\ell-1}\left(\frac{K_m}{k^2}\right)^u p_m^{\ell-j}p_m+ \left(\frac{K_m}{k^2}\right)^\ell S.
      \]
    \item\label{l:4iii} For any $\ell \geq 1$,
      \[
        (p_m + K_m S)^\ell = \sum_{j=0}^{\ell-1} K_m^j p_m^{\ell-j} + K_m^\ell S.
      \]
  \end{enumerate}
\end{lemma}

\begin{proof}
  \begin{enumerate}
 \item The base cases $m=2,3$ follow by straightforward calculations. 
Now assume that $m \geq 3$. If $m$ is odd, from the recursive definition of $p_{m+1}$ and the induction hypothesis, we obtain
  \begin{align*}
    p_{m+1} S &= p_{m+1} + \frac{K_m}{k^2} S + \frac{K_m K_{m-1}}{k^2} S + \frac{K_{m+1}}{k^2} S \\
      &= p_{m+1} + \sum_{u=J_m}^\infty k^{2u-2+2J_{m-1}} x^u + \sum_{u=2^{m-1}}^\infty k^{2u-2+2^{m-1}} x^u \\
      &\quad + \sum_{u=J_{m+1}}^\infty k^{2u-2+2J_m} x^u.
  \end{align*}
  Since $J_{m+1} > 2^{m-1} > J_{m}>2^{m-2}$, it follows that
  \begin{multline*}
   \qquad\quad p_{m+1} S = p_{m+1} + \sum_{u=J_m}^{2^{m-1}-1}k^{2u-2+2J_{m-1}}x^u\\+\sum_{u=2^{m-1}}^{J_{m+1}-1} k^{2u-2+2^{m-1}} x^u + \sum_{u=J_{m+1}}^\infty k^{2u-2+2J_m} x^u.
  \end{multline*}
 Using \Cref{l:3} \Cref{l:3ii} for $m+1$ and the equation $2^{m-1} + J_{m-1} = J_{m+1}$, we arrive at
  \[
    p_{m+1} S = p_{m+1} + \frac{K_{m+1}}{k^2} S.
  \]
  Analogous computations hold if $m$ is even.
  
\item This follows from \Cref{l:4i} by a routine induction argument on $\ell$.

\item By the binomial theorem, the idempotency of $S$, and \Cref{l:4ii}, we obtain
  \begin{align*}
    (p_m + K_m S)^\ell 
      &= \sum_{j=0}^\ell K_m^j p_m^{\ell-j} S \\
      &= \sum_{j=0}^{\ell-1} \sum_{u=0}^{\ell-j-1} \frac{K_m^{u+j}}{k^{2u}} p_m^{\ell-j-u} + \left( \sum_{j=0}^{\ell-1} \frac{1}{k^{2(\ell-j)}} \right) K_m^\ell S + K_m^\ell S \\
      &= \sum_{j=0}^{\ell-1} \sum_{u=0}^{\ell-j-1} \frac{K_m^{u+j}}{k^{2u}} p_m^{\ell-j-u} + K_m^\ell S.
  \end{align*}
  Since, in general, $\sum_{j=0}^{\ell-1} \sum_{u=0}^{\ell-j-1} a_{j,u} = \sum_{u=0}^{\ell-1} \sum_{j=0}^{\ell-u-1} a_{j,u}$, we can rewrite the sum as
  \[
    (p_m + K_m S)^\ell = \sum_{u=0}^{\ell-1} \sum_{j=0}^{\ell-u-1} \frac{K_m^{u+j}}{k^{2u}} p_m^{\ell-j-u} + K_m^\ell S.
  \]
 Consider the change of variables $w = u + j$ for $0 \leq u \leq \ell - 1$ and $0 \leq j \leq \ell - u - 1$. Then $0 \leq w \leq \ell - 1$. For each such $w$, we define the set $A_w$ of integer pairs $(j,u)$ satisfying $w = j + u$ within these ranges:
  \[
    A_w = \{(j,u) \in \mathbb{Z} \times \mathbb{Z} \mid w = j + u, \, 0 \leq u \leq \ell - 1, \, 0 \leq j \leq \ell - u - 1\}.
  \]
  Then,
  \[
    (p_m + K_m S)^\ell = \sum_{w=0}^{\ell-1} \sum_{(j,u) \in A_w} \frac{K_m^w}{k^{2u}} p_m^{\ell-w} + K_m^\ell S.
  \]
  Since $(w,0) \in A_w$ for every $0 \leq w \leq \ell - 1$, it follows that
  \[
    \sum_{(j,u) \in A_w} \frac{K_m^w}{k^{2u}} = K_m^w,
  \]
  which completes the proof.
   \end{enumerate}
\end{proof}

\section{Deformed Jacobsthal and independence Polynomials}\label{s:JacSst}

In this section, we relate independence polynomials to deformed Jacobsthal polynomials by solving recurrence~\eqref{eq:Aro} within the dioid $\DiZ$ using its \emph{reflected} form. We define the reflection operator $R : \N[x] \to \N[x]$ by $R[q(x)] = x^d q(1/x)$, where $d = \deg(q)$. By definition, $R$ is multiplicative and almost additive: for any $p, q \in \N[x]$ with degrees $d \geq e$,
\begin{equation}\label{eq:R}
  R[pq] = R[p]R[q], \quad R[p+q] = R[p] + x^{d-e}R[q].
\end{equation}
We apply $R$ because the log-concavity breaks of $t_{m,n}$ typically occur near its degree. Reflecting the polynomial shifts these breaks toward the constant terms (the ground semiring $\N$), making them easier to analyze. To apply $R$ to recurrence~\eqref{eq:Aro}, we must first determine the degree of $t_{m,n}$.

\begin{lemma}\label{l:1}
  If $n$ is odd and $m \geq 0$, then
  \[
    \deg t_{m,n} = 2^{m-1}(n-1) + \frac{5 \cdot 2^m}{3} + \frac{(-1)^{m+1}}{6} - \frac{1}{2}.
  \]
\end{lemma}

 \begin{proof}
  We proceed by induction on $m$. The base cases $m=0$ and $m=1$ hold since $t_{0,n} = I(P_{n+1};x)$ with $\deg I(P_n;x) = \lceil n/2 \rceil$, and $\deg t_{1,n} = n+2$.
  For the inductive step, applying the hypothesis to \Cref{eq:Aro} gives
  \[
    \deg t_{m+1,n} = 2^m(n-1) + \frac{5 \cdot 2^{m+1}}{3} - 1 + \max \left\{ \frac{(-1)^{m+1}}{3}, \frac{2(-1)^m}{3} \right\}.
  \]
  Regardless of the parity of $m$, the maximum term evaluates exactly to $\frac{(-1)^{m+2}}{6} + \frac{1}{2}$. Substituting this back gives
  \[
    \deg t_{m+1,n} = 2^m(n-1) + \frac{5 \cdot 2^{m+1}}{3} + \frac{(-1)^{m+2}}{6} - \frac{1}{2},
  \]
  which completes the proof.
\end{proof}  

Note that the sequence defined by the differences $(\deg t_{m,n} - 2^{m-1}(n-1))_{m \geq 0}$ corresponds to sequence A081254 in the OEIS \cite{oeisA081254}. As an immediate consequence of \Cref{l:1}, we obtain the reflected version of recurrence \eqref{eq:Aro}.

\begin{theorem}\label{thm:rev}
  Let $n$ be an odd integer. Then, for $m \geq 2$,
  \[
    R[t_{m,n}] = x^{\od(m)} R[t_{m-1,n}]^2 + x^{\ev(m)} R[t_{m-2,n}]^4.
  \]
\end{theorem}
\begin{proof}
  \Cref{l:1} implies the following:
  \begin{enumerate}
    \item If $m$ is odd, then $\deg t_{m,n} = \deg(x t_{m-2,n}^4)$ and
      \[
        \deg(x t_{m-2,n}^4) - \deg(t_{m-1,n}^2) = 1;
      \]
    \item If $m$ is even, then $\deg t_{m,n} = \deg(t_{m-1,n}^2)$ and
      \[
        \deg(t_{m-1,n}^2) - \deg(x t_{m-2,n}^4) = 1.
      \]
  \end{enumerate}
  Combining these degree relations with the multiplicative and almost-additive properties in \eqref{eq:R} of the reflection operator $R$ completes the proof.
\end{proof}

Fix a positive integer $w$. To state our main theorem, we consider the canonical projection morphism $\pi_w \colon \Ring \to \Ring/\langle x^w\rangle$ onto the Bourne factor semiring \cite[p.~97]{Golan}, where $\langle x^w\rangle$ is the ideal generated by $x^w$ in $\Ring$.

\begin{definition}
  If $p, q \in \Ring$ satisfy $\pi_w(p) = \pi_w(q)$, we write
  \[
    p = q \pmod{x^w}.
  \]
\end{definition}
The following lemma establishes the relationship between the initial conditions of recurrence \eqref{eq:Aro} and the geometric series $S = \sum_{j=0}^\infty k^{2j}x^j$.
\begin{lemma}\label{l:2} 
  Let $w$ be a fixed positive integer. Then,
  \[
    R[I(P_{2k}, x)] = kS \pmod{x^w} \quad \text{and} \quad R[I(P_{2k+1}, x)] = S \pmod{x^w}.
  \]
\end{lemma}

\begin{proof}
  From \cite[Prop.~1(iv)]{Hopkins}\footnote{There is an obvious misprint in this reference: the term $n-k-1$ should be $n-k+1$.}, we obtain
  \[
    R[I(P_{2k}, x)] = \sum_{j=0}^k \tbinom{k+j+1}{2j+1} x^j, \quad R[I(P_{2k+1}, x)] = \sum_{j=0}^{k+1} \tbinom{k+j+1}{2j} x^j,
  \]
  where, for any $0 \leq j < w$, the factorial formula for binomial coefficients implies
  \[
    \tbinom{k+j+1}{2j+1} \in \Theta(k^{2j+1}), \quad \tbinom{k+j+1}{2j} \in \Theta(k^{2j}),
  \]
  from which the result follows.
\end{proof}

We are now ready to state the main theorem of this section, which relates the independence polynomials of the spherically symmetric trees $T(2^m 1^n)$ to the deformed Jacobsthal polynomials.
\begin{theorem}\label{thm:main}
  Let $w$ be a fixed positive integer. For every $m \geq 1$ and odd integer $n = 2k+1$, 
  \[
    R[t_{m,n}] = p_m + k^{2^m} x^{J_m} \sum_{j=0}^\infty k^{2j} x^j \pmod{x^w}.
  \]
\end{theorem}

\begin{proof}
  We proceed by induction on $m$. For $m=1$, \Cref{l:2} proves that
  \begin{align*}
    R[t_{1,n}] &= R[I(P_{2k+3}; x)] 
              = S \pmod{x^w} \\
               &= p_1 + K_1 S.
  \end{align*}
  For $m=2$, applying \Cref{thm:rev}, \Cref{l:2,l:SIdem} gives
  \begin{align*}
    R[t_{2,n}] &= R[t_{1,n}]^2 + x R[t_{0,n}]^4 
               = S + \sum_{j=0}^{\infty} k^{2j+4} x^{j+1} \pmod{x^w} \\
               &= 1 + \sum_{j=1}^{\infty} k^{2j+2} x^j 
               = p_2 + K_2 \sum_{u=0}^\infty k^{2u} x^u.
  \end{align*}
Now assume $m \geq 2$. If $m$ is even, then the induction hypothesis, \Cref{thm:rev}, and \Cref{l:4} \Cref{l:4iii} ensure that
  \begin{align*}
    R[t_{m+1,n}] &= R[t_{m-1,n}]^4 + x R[t_{m,n}]^2 \\
                 &= (p_{m-1} + K_{m-1} S)^4 + x(p_m + K_m S)^2 \pmod{x^w} \\
                 &= \sum_{j=0}^3 K_{m-1}^j p_{m-1}^{4-j} + K_{m-1}^4 S + x p_m^2 + x K_m p_m + K_{m+1} S,
  \end{align*}
  where $K_{m-1}^4 S \abs K_{m+1} S$ since $K_{m-1}^4 = x K_{m+1}$ and $xS \abs S$.
Thus,
  \begin{align*}
    R[t_{m+1,n}] &= x \sum_{j=0}^1 K_m^j p_m^{2-j} + \sum_{j=0}^3 K_{m-1}^j p_{m-1}^{4-j} + K_{m+1} S \pmod{x^w}\\
                 &= x \sum_{j=0}^1 K_m^j p_m + \sum_{j=0}^3 K_{m-1}^j p_{m-1} + K_{m+1} S \quad {\text{(by \Cref{c:NoPow})}} \\
                 &= p_{m+1} + \frac{K_{m+1}}{k^4} + K_{m+1} S \quad{\text{(by \Cref{l:Alst})} }\\
                 &= p_{m+1} + K_{m+1} S.
  \end{align*}
  The case where $m$ is odd is handled analogously.
\end{proof}

To use \Cref{thm:main} to reveal breaks in the log-concavity of the polynomials $t_{m,n}$, we must first determine the degree of the deformed Jacobsthal polynomials.
\begin{lemma}\label{l:degp_m}
  For all $m \geq 1$, $\deg p_m = J_m - 1$.
\end{lemma}
\begin{proof}
  Since $p_1 = p_2 = 1$ and $J_1 - 1 = J_2 - 1 = 0$, the result holds for the base cases $m=1, 2$. Assume $m \geq 2$. If $m+1$ is even, then
  \[
    p_{m+1} = p_m + K_m p_{m-1} + K_m K_{m-1} p_{m-1},
  \]
  and applying the induction hypothesis yields
  \[
    \deg p_{m+1} = \max \{ J_m - 1, J_m + J_{m-1} - 1, \underbrace{J_m + 2J_{m-1}}_{J_{m+1}} - 1 \} = J_{m+1} - 1.
  \]
  The odd case follows symmetrically.
\end{proof}

The next lemma shows that the deformed Jacobsthal polynomials exhibit breaks in log-concavity in the following sense. From the definition of these polynomials, we can write $p_m = \sum_{j=0}^{J_m-1} a_j(k) x^j$, where, for all $0 \leq j < J_m$, each $a_j(k)$ is a polynomial in $k$ with natural coefficients (in fact, a power of $k$). We say that $p_m$ has \emph{an asymptotic break in log-concavity} at index $u$ if
\begin{equation}\label{eq:asympBrk}
  \deg(a_{u-1}(k)) + \deg(a_{u+1}(k)) > 2 \deg(a_u(k)).
\end{equation}

\begin{lemma}\label{l:breaks}
  For all $m \geq 3$, $p_m$ has at least $J_{m-3}$ asymptotic breaks in log-concavity.
\end{lemma}

\begin{proof}
  The base cases hold directly: $p_3 = 1 + k^2 x + k^4 x^2$ has $J_0 = 0$ asymptotic breaks in log-concavity, and $p_4 = 1 + k^2 x + k^4 x^2 + k^8 x^3 + k^{12} x^4$ breaks log-concavity at index $2$, giving $J_1 = 1$ break. For odd $m > 4$, we have $p_m = p_{m-2} + K_{m-2} p_{m-2} + K_{m-2}^2 p_{m-1}$. By \Cref{l:degp_m}, these three summands have disjoint degree ranges ($0$ to $J_{m-2}-1$, $J_{m-2}$ to $2J_{m-2}-1$, and $2J_{m-2}$ to $J_m-1$, respectively), so their breaks do not interfere. Therefore, by the induction hypothesis, $p_m$ inherits at least $2J_{m-5} + J_{m-4} = J_{m-3}$ asymptotic breaks in log-concavity. The case where $m$ is even is analogous.
\end{proof}

\begin{corollary}\label{c:final}
  For any $m \geq 3$ and sufficiently large odd integer $n$, the independence polynomial of the spherically symmetric tree $T(2^m 1^n)$ has at least $J_{m-3}$ breaks in log-concavity.
\end{corollary}

\begin{proof}
  Let $R[t_{m,n}] = \sum_{j=0}^\infty c_j x^j$ and $p_m = \sum_{j=0}^{J_m-1} a_j x^j$, where $c_j, a_j \in \N[k]$. Applying \Cref{thm:main} with $w \geq J_m$ yields $c_j \in \Theta(a_j)$ for $j = 0, \dots, J_m - 1$. Consequently, every asymptotic break in log-concavity of $p_m$ guaranteed by \Cref{l:breaks} yields a break in log-concavity of $R[t_{m,n}]$ for sufficiently large $k$, since \Cref{eq:asympBrk} implies $c_{u-1} c_{u+1} - c_u^2 \in \Theta(a_{u-1} a_{u+1})$.
\end{proof}
This corollary affirmatively answers Question~3.2 in \cite{Galvin}, since $J_{m+2} > m$ for all nonnegative integers $m$.

% \bibliographystyle{siamplain}
% \bibliography{refs}
\end{document}